\documentclass[preprint,
12pt]{elsarticle}

\usepackage[nodots,nocompress]{numcompress}
\usepackage{amssymb,amsmath,amsthm
}
\usepackage{pstricks,multido,pst-plot}
\usepackage[fleqn,tbtags]{mathtools}

\newtheorem{thm}{Theorem}[section]
\newtheorem{lemma}[thm]{Lemma}
\newtheorem{cor}[thm]{Corollary}

\newtheorem{prop}[thm]{Proposition}
\theoremstyle{definition}
\newtheorem{defn}{Definition}

\newtheorem{ex}{Example}
\newtheorem*{rem}{Remark}

\numberwithin{equation}{section}
\numberwithin{thm}{subsection}

\newcommand{\supp}{\mathop\mathrm{supp}\nolimits}

\newcommand{\intr}{\mathop\mathrm{Int}\nolimits}
\newcommand{\graph}{\mathop\mathrm{gr}\nolimits}
\newcommand{\hd}{\mathop\mathrm{dim_H}\nolimits}

\journal{Internat. J. Approx. Reason.}

\begin{document}

\begin{frontmatter}



\title{Essential closures and supports of multivariate copulas}


\author[chula]{P.~Ruankong}
\ead{ruankongpol@gmail.com}
\author[chula,cem]{S.~Sumetkijakan\corref{cor} 
}
\ead{songkiat.s@chula.ac.th}

\address[chula]{Department of Mathematics and Computer Science, Faculty of Science, \mbox{Chulalongkorn University}, \mbox{Phyathai Road}, Patumwan, Bangkok 10330, Thailand}
\address[cem]{Centre of Excellence in Mathematics, CHE, Si Ayutthaya Rd., Bangkok 10400, Thailand}
\cortext[cor]{Corresponding author}

\begin{abstract}
We generalize the notion of essential closures which is used in formulating a geometric necessary condition for a set to be the support of a multivariate copula. Furthermore, in some special cases, we derive an explicit formula of the support in terms of essential closures and obtain a stronger necessary condition.
\end{abstract}


\end{frontmatter}


\section{Introduction}\label{intro}

For decades, supports of doubly stochastic measures have been studied by many mathematicians. A fundamental aspect of doubly stochastic measures is their support which is exactly where the probability mass is concentrated. Of particular interest is the study of supports of extreme doubly stochastic measures. A handful of necessary conditions and some characterizations have been obtained (see, for example, \cite{Char,Pres,Hairpin}).

We study the supports of multivariate copulas, or equivalently multivariate stochastic measures, from a different approach. First, we introduce the notion of \lq\lq essential closures,\rq\rq which is a generalization of the topological closure and a generalization of the essential closure introduced in \cite{Essential}. Essential closures seem to be a suitable tool to study the geometry of the supports of multivariate copulas, or equivalently multivariate stochastic measures. We obtain geometric neccesary conditions via \lq\lq essential closedness,\rq\rq which is a closedness with respect to an essential closure. Surprisingly, in some special cases, it turns out that we can derive an explicit formula of the support in terms of essential closures. One such case is the case of bivariate complete dependence copulas. And if we assume that the random variables are uniform on $[0,1]$, there is a measure-preserving transformation on $[0,1]$ connecting the two random variables. It has been observed that the graph of such a function and the support of the corresponding copula are closely related. For instance, it has been shown in \cite{Independence} that the mass of a copula is concentrated on the graph of a corresponding function ($V_C(\graph{f})=1$). In our work, we obtain that the support of such copula is an essential closure of the graph of a \lq\lq refinement\rq\rq of the function. In fact, we obtain a more general result in higher dimensions.

This paper is organized as follows. In section 2, we recall basic definitions and properties of related terms. In section 3, we introduce the notion of essential closures and state some of their properties, then we derive the main results. Finally, in section 4, we give proofs of the properties of essential closures.
\section{Preliminaries}

\subsection{Multivariate copulas and multivariate stochastic measures}

\begin{defn} A \emph{$k$-copula} is a function $C \colon [0,1]^k \rightarrow [0,1]$ satisfying
\begin{enumerate}
	\item $C(u_1,\dots,u_{i-1},0,u_{i+1},\dots,u_k)=0$,
	\item $C(1,\dots,1,u_i,1,\dots,1)=u_i$, and
	\item $C$ is $k$-increasing, i.e., for each hyperrectangle $B=\displaystyle\times_{i=1}^k{[x_i,y_i]} \subseteq [0,1]^k$,
	\begin{displaymath} V_C(B)=\sum_{z \in \times_{i=1}^k\{x_i,y_i\}} (-1)^{N(z)} C(z) \ge 0,\end{displaymath} where $N(z)$ denotes the size of the set $\{i\colon z_i=x_i\}$.
\end{enumerate}
\end{defn}

The set function $V_C$ can be extended uniquely to a Borel probability measure on $[0,1]^k$ and is often called the \emph{C-volume}. In fact, $V_C$ is a \emph{$k$-stochastic measure}, i.e., it pushforwards to $1$-dimensional Lebesgue measure on each axis. Moreover, the \emph{support of $C$} is defined to be the support of the measure $V_C$.

\begin{thm} [Sklar's theorem] Let $X_1,\dots,X_k$ be random variables on a common probability space. Let $H$ be their joint distribution and $F_i$ be the margin of $X_i$. Then there is a $k$-copula $C$ such that \begin{center}$H(x_1,\dots x_k)=C(F_1(x_1),\dots,F_n(x_k))$.\end{center} Moreover, if $X_1,\dots,X_k$ are continuous, then $C$ is unique.
\end{thm}

It is worth mentioning that a multivariate copula can also be viewed as a joint distribution of uniform $[0,1]$ random variables. 

\subsection{Multivariate complete dependence}

\begin{defn} Given two nonempty sets $\mathbf{A}$ and $\mathbf{B}$ of random variables on a common probability space, we say that $\mathbf{A}$ is \emph{completely dependent} on $\mathbf{B}$ (viewed as a random vector) if, for every $X \in \mathbf{A}$, there exists a Borel measurable function $f$ such that $X=f(\mathbf{B})$ almost surely.
\end{defn}

\begin{defn} A $k$-copula $C$ is said to be a \emph{bipartite dependence $k$-copula} if $C$ is a $k$-copula of a collection of continuous random variables $\{X_1,\dots,X_k\}$ which can be partitioned into two sets so that one set is completely dependent on the other.\label{multicd}
\end{defn}

\begin{rem} One can always take $X_1,\dots,X_k$ in Definition \ref{multicd} to be uniformly distributed on $[0,1]$.
\end{rem}

\begin{ex} Let $f \colon [0,1] \rightarrow [0,1]$ be such that $f(x)=x$ if $x$ is irrational and $f(x)=1-x$ otherwise. Let $U$ be a uniform $[0,1]$ random variable and $V=f(U)$ almost surely. So $V$ is also a uniform $[0,1]$ random variable since $f$ is measure-preserving.
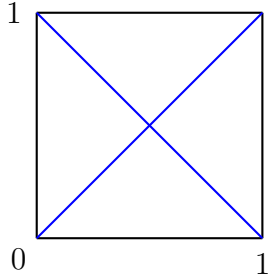
\begin{figure}[ht]
\psset{xunit=3cm,yunit=3cm}
\begin{center}
\begin{pspicture*}(-0.2,-0.2)(1.2,1.2)
\psline[linecolor=black](0,1)(1,1)
\psline[linecolor=black](1,0)(1,1)
\psline[linecolor=black](0,0)(0,1)
\psline[linecolor=black](0,0)(1,0)
\psplot[linecolor=blue,plotpoints=400]{0}{1}{x}
\psplot[linecolor=blue,plotpoints=400]{0}{1}{1 x sub}
\uput{0.2}[225](0,0){0}
\uput{0.2}[270](1,0){1}
\uput{0.2}[180](0,1){1}
\end{pspicture*}
\end{center}
  \caption[ ]{the closure of the graph of $f$ in Example \ref{ex:1}} 
 \label{fig:ex}
\end{figure}
Notice that the closure of the graph of $f$, shown in Figure \ref{fig:ex}, is the union of the lines $y=x$ and $y=1-x$ in $[0,1]^2$. But since the copula of $(U,V)$ is the 2-dimensional Min copula $M_2$, its support is the line $y=x$ in $[0,1]^2$. Therefore, topological closure is not a suitable tool. A suitable tool should at least be able to get rid of the line $y=1-x$ which is not \lq\lq essential\rq\rq to the function $f$. We construct such a tool in the next section.\label{ex:1}
\end{ex}

\section{Essential closures and the main results}
\subsection{Essential closures}

In the sequel, we call a subspace $W$ of a Euclidean space spanned by standard basis elements \emph{a standard subspace}. Moreover, we denote $d$-dimensional Lebesgue measure and outer measure by $\lambda_d$ and $\lambda_d^*$, respectively. Also, $\lambda_0=\lambda_0^*$ denotes the counting measure. In addition, $\pi_W$ denotes the orthogonal projection onto the subspace $W$ and $\pi_j$ denotes the orthogonal projection onto the $j$-th axis. Note that $\pi_{\{0\}}(A) =\{0\}$ if $A$ is nonempty, otherwise $\pi_{\{0\}}(A) =\emptyset$. Finally, $\mathfrak{N}(x)$ denotes the set of open neighborhoods of the point $x$.

\begin{defn} Let $A$ be a subset of $\mathbb{R}^k$ and $0 \le d \le k$ be an integer. Then $x\in \widehat{A}^{^{\,d}}$ if and only if, for each $G \in \mathfrak{N}(x)$, there exists a $d$-dimensional standard subspace $W$ such that $\lambda_d^*(\pi_W(G \cap A)) >0.$ The set $\widehat{A}^{^{\,d}}$ is called the \emph{$d$-essential closure of $A$}. Moreover, $A$ is said to be \emph{$d$-essentially closed} if $A=\widehat{A}^{^{\,d}}.$ \end{defn}


The notion of essential closures is a generalization of the topological closure. To be precise, the $0$-essential closure and the topological closure coincide as is easily verified. The following properties of essential closures will be proved in the last section.

\begin{prop} Let $A,B \subseteq \mathbb{R}^k$ and $0 \le d \le k$ be an integer. Then the following hold.
\begin{enumerate}
	\item $\widehat{A}^{^{\,d}}$ is closed.
	\item Let $e \in \mathbb{Z}$ such that $0 \le e < d$. Then $\widehat{A}^{^{\,d}} \subseteq \widehat{A}^{^{\,e}}$. In particular, $\widehat{A}^{^{\,d}} \subseteq \overline{A}$.
	\item $\overline{\intr A} \subseteq \widehat{A}^{^{\,d}}$.
	\item $\widehat{A}^{^{\,d}} \subseteq \widehat{B}^{^{\,d}}$.
	\item $\widehat{A \cup B}^{d} = \widehat{A}^{^{\,d}} \cup \widehat{B}^{^{\,d}}$.
	\item $\widehat{A}^{^{\,d}} =\emptyset$ if and only if $\lambda_d^*(\pi_W(A)) =0$ for every $d$-dimensional standard subspace $W$.
	\item $\widehat{A}^{^{\,d}}$ is $d$-essentially closed.
\end{enumerate}
\label{ess-prop}
\end{prop}



For the rest of this section, all sets are Borel measurable, hence their projection images are Lebesgue measurable.


\subsection{Supports of multivariate copulas}






\begin{thm} For every $k$-copula $C$, its support $\supp C$ is $1$-essentially closed.\label{strong}\end{thm}
\begin{proof} Observe that the set of $V_C$-measure zero open sets is an open cover of $(\supp C)^c$. Since Euclidean spaces are Lind\"{e}lof, there exists $\{G_1,G_2,\dots\}$, a countable subcover of $(\supp C)^c$. Then \begin{center}$V_C((\supp C)^c) \le \displaystyle\sum_{i=1}^{\infty}V_C(G_i)=0$.\end{center} For any Borel set $A \subseteq [0,1]^k$, write \begin{center}$V_C(A)=V_C(A \cap \supp C) + V_C(A \cap (\supp C)^c)$.\end{center} Observe that $V_C(A \cap (\supp C)^c) \le V_C((\supp C)^c) =0.$ Consequently, we have $V_C(A)=V_C(A \cap \supp C)$.

Since $\supp C$ is closed, it follows from Proposition \ref{ess-prop}(2) that $\widehat{\supp C}^{_1} \subseteq \overline{\supp C} =\supp C$. It is left to show that $\supp C \subseteq \widehat{\supp C}^{_1}$. By the definition of essential closures, if $x \notin \widehat{\supp C}^{_1}$, then there exists $G\in \mathfrak{N}(x)$ such that $\lambda_1(\pi_1(G \cap \supp C))=0$. Since $V_C$ is $k$-stochastic,
\begin{align*}
V_C(G)&=V_C(G \cap \supp C)\\
&\le V_C(\pi_1(G \cap \supp C) \times [0,1]^{k-1})\\
&= \lambda_1(\pi_1(G \cap \supp C))\\
&=0.
\end{align*} So $x \notin \supp C$.
\end{proof}

\begin{lemma} Let $C$ be a $k$-copula. Then, for any open set $G$ intersecting the support of $C$, the intersection cannot be a subset of a $(k-1)$-dimensional hyperplane perpendicular to an axis. \label{hyper}
\end{lemma}

\begin{proof} Let $W \subseteq \mathbb{R}^k$ be a $(k-1)$-dimensional hyperplane perpendicular to an axis. It suffices to show that $W \cap [0,1]^k$ has $V_C$-measure zero. Suppose $W$ is perpendicular to the $i$-th axis at a point $x \in [0,1]$. Since $V_C$ is $k$-stochastic, \begin{center}$V_C(W \cap [0,1]^k)=\lambda_1(\{x\})=0.$\end{center} Therefore, any Borel subset of $W$ inside $[0,1]^k$ also has $V_C$-measure zero.
\end{proof}

Theorem \ref{strong} and Lemma \ref{hyper} give geometric necessary conditions for a set to be the support of a multivariate copula. However, these necessary conditions are not sufficient even if we assume that the set has full projection image on each axis. For example, a hairpin-like set is $1$-essentially closed with full projection image on each axis but not always the support of a $2$-copula as mentioned in \cite{Hairpin}.

\begin{ex} As a consequence of Lemma \ref{hyper}, the set shown in Figure \ref{fig:cex} cannot be the support of a copula.

\begin{figure}[ht]
\psset{xunit=3cm,yunit=3cm}
\begin{center}
\begin{pspicture*}(-0.2,-0.2)(1.2,1.2)
\psline[linecolor=black](0,1)(1,1)
\psline[linecolor=black](1,0)(1,1)
\psline[linecolor=black](0,0)(0,1)
\psline[linecolor=black](0,0)(1,0)
\psplot[linecolor=blue,plotpoints=400]{0}{0.5}{1 x sub}
\psline[linecolor=blue,plotpoints=400](0.5,0)(0.5,0.5)
\psline[linecolor=blue,plotpoints=400](0.5,0.5)(1,0.5)
\uput{0.2}[225](0,0){0}
\uput{0.2}[270](1,0){1}
\uput{0.2}[180](0,1){1}
\end{pspicture*}
\end{center}
  \caption[ ]{a set that is not the support of a copula} 
  \label{fig:cex}
\end{figure}
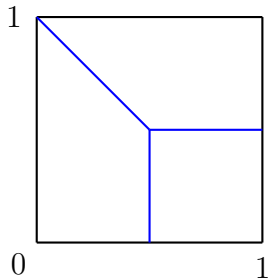
\end{ex}

In the next subsection, we explore a special case in which it is possible to determine the supports via essential closures.


\subsection{Supports of bipartite dependence multivariate copulas}


\begin{defn} Let $A \subseteq \mathbb{R}^k$ and $\sigma$ be a permutation on $\{1,2,\dots,k\}.$ Define \emph{the coordinate permutation of $A$ with respect to $\sigma$} by \begin{center}$A_{\sigma}= \{(x_{\sigma(1)},x_{\sigma(2)},\dots,x_{\sigma(k)}) \colon (x_1,x_2,\dots,x_k) \in A\}$.\end{center}
\end{defn}

\begin{lemma} Let $\sigma$ be a permutation on $\{1,2,\dots,k\}.$ Let $C$ be the $k$-copula of continuous random variables $X_1,X_2,\dots,X_k$ and let $C_{\sigma}$ be the $k$-copula of $X_{\sigma(1)},X_{\sigma(2)},\dots,X_{\sigma(k)}$. Then $\supp C_{\sigma} = (\supp C)_{\sigma}$.\label{permute}
\end{lemma}

\begin{proof} Observe that, for every open set $G \subseteq \mathbb{R}^k$, \begin{align*} V_C(G)&=P((X_1,\dots,X_k) \in G)\\
&=P((X_{\sigma(1)},\dots,X_{\sigma(k)}) \in G_{\sigma})\\
&=V_{C_{\sigma}}(G_{\sigma}). \end{align*}
As a consequence, we have $x \in \supp C$ if and only if $x_{\sigma} \in \supp C_{\sigma}.$
Therefore, $\supp C_{\sigma} = (\supp C)_{\sigma}$.
\end{proof}
In view of Lemma \ref{permute}, in the case of bipartite dependence $k$-copulas, we may rearrange the random variables so that, for some $i$, each random variable $X_j$, $j\in\{i+1,\dots,k\}$, is completely dependent on the random vector $(X_1,\dots,X_i)$.


\begin{defn} Let $\mathcal{F}=(f_1,\dots,f_m)$ be a tuple of functions from $D$ into $R$. Define \emph{the hypergraph of $\mathcal{F}$} by \begin{center}$\graph \mathcal{F} = \{(x,f_1(x),\dots,f_m(x)) \in D\times R^m\colon x \in D\}$.\end{center}
\end{defn}

\begin{defn} A function $\mathcal{F} \colon [0,1]^n \rightarrow [0,1]^m$ with Borel coordinate functions is said to have \emph{a Borel essential refinement} if there exists a function $\mathcal{F}^* \colon[0,1]^n \rightarrow [0,1]^m$ with Borel coordinate functions such that each corresponding pair of coordinate functions of $\mathcal{F}^*$ and $\mathcal{F}$ are equal almost everywhere and, for any open set $G \subseteq \mathbb{R}^{n+m}$, 

\begin{center} $\lambda_n(\pi_{W_0}(G \cap \graph \mathcal{F^*}))=0$ implies that $\forall j>n,\lambda_1(\pi_j(G \cap \graph \mathcal{F}^*))=0$\end{center}
where $W_0$ is the subspace spanned by the first $n$ standard basis elements.

\end{defn}

\begin{rem} A function with Borel coordinate functions is Borel measurable. Hence its graph is a Borel measurable set.
\end{rem}

\begin{lemma} Every function $\mathcal{F} \colon [0,1]^n \rightarrow [0,1]^m$ with Borel coordinate functions has a Borel essential refinement.
\end{lemma}

\begin{proof} The idea is to redefine $\mathcal{F}$ on a set of Borel measure zero. Let $W_0$ be the subspace spanned by the first $n$ standard basis elements. Let \begin{center}$V=\displaystyle \bigcup_{\alpha \in \Lambda} V_{\alpha}$\end{center} where the union is taken over all open sets $V_{\alpha}$ such that $\lambda_n(\pi_{W_0}(V_{\alpha} \cap \graph \mathcal{F})) = 0$ while $\lambda_1(\pi_j(V_{\alpha} \cap \graph \mathcal{F})) > 0$ for some $j >n$. By Lind\"{e}lof property of Euclidean spaces, there exists a countable subcollection: $\{V_1,V_2,\dots\}$ whose union is $V$. Then we have \begin{center} $\lambda_n(\pi_{W_0}(V \cap \graph \mathcal{F})) \le \displaystyle\sum_{i=1}^{\infty}\lambda_n(\pi_{W_0}(V_i \cap \graph \mathcal{F}))=0.$ \end{center} Thus there exists a Borel measure zero set $B \subseteq \mathbb{R}^{n}$ such that $\pi_{W_0}(V \cap \graph \mathcal{F}) \subseteq B$. Define $\mathcal{F}^*=\mathcal{F}-\mathcal{F}\cdot\chi_B$, i.e., $\mathcal{F}$ is redefined on $B$ to be identically zero. Consequently, the coordinate functions of $\mathcal{F}^*$ are Borel and each corresponding pair of coordinate functions of $\mathcal{F}^*$ and $\mathcal{F}$ are equal almost everywhere.

Suppose there is an open set $G \in \mathbb{R}^{n+m}$ such that $\lambda_n(\pi_{W_0}(G \cap \graph \mathcal{F}^*)) =0$ while $\lambda_1(\pi_j(G \cap \graph \mathcal{F}^*)) >0$ for some $j >n$. Observe that $\pi_{W_0}(G \cap \graph \mathcal{F})$ and $\pi_{W_0}(G \cap \graph \mathcal{F}^*)$ differ by a subset of $B$, which is a Borel measure zero set. Moreover, $\pi_j(G \cap \graph \mathcal{F})$ contains $\pi_j(G \cap \graph \mathcal{F}^*)-\{0\}$ due to the redefining. Thus $\lambda_n(\pi_{W_0}(G \cap \graph \mathcal{F})) =0$ while $\lambda_1(\pi_j(G \cap \graph \mathcal{F})) >0$. Hence $G \subseteq V$. This means that the points inside $G \cap \graph \mathcal{F}$ were redefined, i.e., $\lambda_1(\pi_j(G \cap \graph \mathcal{F}^*))=\lambda_1(\{0\})=0$, a contradiction. Therefore, $\mathcal{F}^*$ is a Borel essential refinement of $\mathcal{F}$.
\end{proof}

\begin{rem} Our results require that the random variables are uniform $[0,1]$. This is by no means restrictive since, for given continuous random variables $X_1,\dots,X_k$, each $U_i=F_{X_i}(X_i)$ is uniform on $[0,1]$ for each $i$. Moreover,  \begin{center}$C_{X_1,\dots,X_k}=C_{U_1,\dots,U_k}$.\end{center} So, it suffices to study only uniform $[0,1]$ random variables. 
\end{rem}

Now, we are ready to derive the main results. As mentioned in the introduction, the support of a complete dependence copula is closely related to the graph of a corresponding Borel function. In the next theorem, we derive an explicit formula of such a support in a finite dimension in terms of the essential closures.

\begin{thm} Let $U_1,U_2,\dots,U_{n+m}$ be uniform $[0,1]$ random variables and $C$ be their multivariate copula. Let $\mathbf{U}$ denote the random vector $(U_1,U_2,\dots,U_n)$. Suppose that $\lambda_n \ll V_{C_{\mathbf{U}}} \ll \lambda_n$ if $n \ge 2$. If, for each $i\in \{1,2,\dots,m\}$, $U_{n+i}$ is completely dependent on $\mathbf{U}$, i.e., there exist Borel functions $f_i\colon [0,1]^n \rightarrow [0,1]$ such that $U_{n+i} = f_i(\mathbf{U})$ almost surely, then \begin{center} $\supp C = \widehat{\graph\mathcal{F}^*}^{\scriptscriptstyle n}$ \end{center} where $\mathcal{F}^*=(f^*_1,\dots,f^*_m)$ is a Borel essential refinement of $\mathcal{F}=(f_1,\dots,f_m)$.\label{explicit}
\end{thm}

\begin{proof} First of all, since $f^*_i=f_i$ almost everywhere, $f^*_i(\mathbf{U})=f_i(\mathbf{U})$ almost surely. Hence $C=C_{\mathbf{U},\mathcal{F}^*(\mathbf{U})}$. Let $W_0$ be the subspace spanned by the first $n$ standard basis elements. Observe that, for every open set $G \subseteq \mathbb{R}^{n+m}$,
\begin{center}
$V_C(G)=P((\mathbf{U},\mathcal{F}^*(\mathbf{U})) \in G)=P(\mathbf{U} \in \pi_{W_0}(G\cap \graph \mathcal{F}^*)).$
\end{center}
Moreover, since $\lambda_n \ll V_{C_{\mathbf{U}}} \ll \lambda_n$ if $n \ge 2$, we have \begin{center} $P(\mathbf{U} \in \pi_{W_0}(G\cap \graph \mathcal{F}^*))>0 $ if and only if $\lambda_n(\pi_{W_0}(G \cap \graph \mathcal{F}^*))>0.$\end{center}
By the definition of essential closures, $\supp C \subseteq \widehat{\graph\mathcal{F}^*}^{\scriptscriptstyle n}$.

Conversely, suppose there exists an $n$-dimensional standard subspace $W$ such that $\lambda_n(\pi_{W}(G \cap \graph \mathcal{F}^*))>0$ while $\lambda_n(\pi_{W_0}(G \cap \graph \mathcal{F}^*))=0$. Then there exists $j >n$ such that $\lambda_1(\pi_j(G \cap \graph \mathcal{F}^*))>0$. These contradict the fact that $\mathcal{F}^*$ is a Borel essential refinement. Therefore, for every $n$-dimensional standard subspace $W$, $\lambda_n(\pi_{W}(B \cap \graph \mathcal{F}^*))>0$ implies $\lambda_n(\pi_{W_0}(B \cap \graph \mathcal{F}^*))>0$. Thus $\widehat{\graph\mathcal{F}^*}^{\scriptscriptstyle n} \subseteq \supp C$.
\end{proof}

Note that each Borel function $f_i^*$ in Theorem \ref{explicit} is measure-preserving since it is a transformation between two uniform random variables.

As a consequence of Theorem \ref{explicit} and Proposition \ref{ess-prop}, the support of a bipartite dependence multivariate copula is essentially closed. The next result gives a geometric interpretation derived directly from the essential closedness of the support.

\begin{cor} Let $C$ be the bipartite dependence multivariate copula defined in Theorem \ref{explicit}. Then $\supp C$ is an $n$-essentially closed set. In particular, $\supp C$ has local Hausdorff dimension at least $n$.
\end{cor}

\begin{proof} Since $\supp C$ can be written as an $n$-essential closure of a set, it is $n$-essentially closed by Proposition \ref{ess-prop}(7). Suppose there is an open set $G \subseteq \mathbb{R}^{n+m}$ such that $G \cap \supp C \neq \emptyset$ and the Hausdorff dimension $\hd(G \cap \supp C)<n$. Then, for each $n$-dimensional standard subspace $W$, we have that $\hd(\pi_W(G \cap \supp C)) <n$ which implies $\lambda_n(\pi_W(G \cap \supp C)) =0$, contradicting the fact that $\supp C$ is $n$-essentially closed.
\end{proof}

\begin{ex} This example demonstrates a way to extract a function, connecting the two uniform $[0,1]$ random variables, from the support of a shuffle of $M_2$.

\begin{figure}[ht]
\psset{xunit=3cm,yunit=3cm}
\begin{center}
\begin{pspicture*}(-0.2,-0.2)(1.2,1.2)
\psline[linecolor=black](0,1)(1,1)
\psline[linecolor=black](1,0)(1,1)
\psline[linecolor=black](0,0)(0,1)
\psline[linecolor=black](0,0)(1,0)
\psline[linecolor=black,linestyle=dotted](0.2,0)(0.2,1)
\psline[linecolor=black,linestyle=dotted](0.7,0)(0.7,1)
\psline[linecolor=black,linestyle=dotted](0,0.5)(1,0.5)
\psline[linecolor=black,linestyle=dotted](0,0.8)(1,0.8)
  \psplot[linecolor=blue,plotpoints=400]{0}{0.2}{x 0.8 add}
  \psplot[linecolor=blue,plotpoints=400]{0.2}{0.7}{0.7 x sub}
 \psplot[linecolor=blue,plotpoints=400]{0.7}{1}{x 0.7 sub 0.5 add}
\uput{0.2}[225](0,0){0}
\uput{0.2}[270](1,0){1}
\uput{0.2}[180](0,1){1}
\end{pspicture*}
\end{center}
  \caption[ ]{the support of a shuffle of $M_2$} 
  \label{fig:shuffle}
\end{figure}
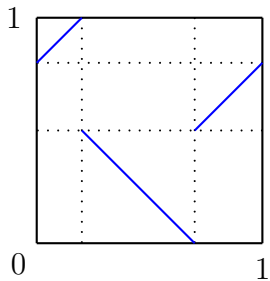 

Observe that the support of a shuffle of $M_2$ is \lq\lq almost\rq\rq a function. In fact, removing a few points from the support of a shuffle of $M_2$ gives us the graph of a function. This can be done in many ways. One way to do it is to simply remove the rightmost point from each linear piece except the last piece. One can see that the remaining set is the graph of a function whose explicit formula can be derived. Then, by Theorem \ref{explicit} and the fact that shuffles of $M_2$ are uniquely determined by their supports, it ensures that we have the right function since the $1$-essential closure of the graph of that function is equal to the support we started with.
\end{ex}

\begin{ex} Let $M_3$ denote the 3-dimensional Min copula. Then $M_3=C_{U,U,U}$ for some uniform $[0,1]$ random variable $U$. Choose $\mathcal{F}=(id_{[0,1]},id_{[0,1]})$. Hence $M_3 = C_{U,\mathcal{F}(U)}$. Moreover, $\graph{\mathcal{F}}$ is $1$-essentially closed. Therefore, \begin{center}$\supp M_3 = \graph{\mathcal{F}} = \{(x,x,x) \colon x \in [0,1] \}$\end{center} which is the main diagonal of the unit cube $[0,1]^3$. Similarly, $\supp M_n$ is the main diagonal of the hypercube $[0,1]^n$.

Notice that, we need to carefully choose a bipartition in order to apply Theorem \ref{explicit}. For example, we can view the last random variable as being completely dependent on the first two. In this case, there are many functions which connect them, e.g., $f_1(x,y)=\frac{x+y}{2}$, $f_2(x,y)=\sqrt{xy}$, etc. Certainly, we cannot apply the theorem and one reason is because the copula of the first two random variables is $M_2$ whose induced measure is not absolutely continuous with respect to $\lambda_2$.
\end{ex}

\section{Proof of basic properties of essential closures}


In Proposition \ref{ess-prop}, each property follows directly from the definition of the essential closures, the properties of orthogonal projections and the properties of outer measures.

\begin{proof} 1. It suffices to show that the closure of $\widehat{A}^{^{\,d}}$ is a subset of itself. If $x$ is in the closure of $\widehat{A}^{^{\,d}}$ and $G \in \mathfrak{N}(x)$, then $G \cap \widehat{A}^{^{\,d}} \neq \emptyset$. Choose $y \in G \cap \widehat{A}^{^{\,d}}$. Thus $G \in \mathfrak{N}(y)$ and $y \in \widehat{A}^{^{\,d}}$. Therefore $\lambda_d^*(\pi_W(G \cap A)) >0$ for some $d$-dimensional standard subspace $W$. Hence $x \in \widehat{A}^{^{\,d}}$.


2. Let $0 \le e < d$ be integers. If $x \notin \widehat{A}^{^{\,e}}$, then there exists $G \in \mathfrak{N}(x)$ such that, for each $e$-dimensional standard subspace $V$, $ \lambda_e^*(\pi_V(G \cap A))=0$. Thus $\lambda_d^*(\pi_W(G \cap A))=0$ for each $d$-dimensional standard subspace $W$. Therefore $x \notin \widehat{A}^{^{\,d}}$.


3. Assume that $\intr A$ is not empty. Let $x \in \intr A$. Then for any $G \in \mathfrak{N}(x)$, $G \cap \intr A
$ is a nonempty open subset of $\mathbb{R}^n$. Therefore, given a $d$-dimensional standard subspace $W$, $\pi_W(G \cap A)$ contains a nonempty open subset of $W$. Thus $\lambda_d^*(\pi_W(G \cap A)) >0$. Hence $\intr A \subseteq \widehat{A}^{^{\,d}}$. Taking closure on both sides, we obtain the desired result.


4. If $x \in \widehat{A}^{^{\,d}}$, then for any $G \in \mathfrak{N}(x)$, there exists a $d$-dimensional standard subspace $W$ such that $\lambda_d^*(\pi_W(G \cap A)) >0.$ Thus we have $\lambda_d^*(\pi_W(G \cap B)) \ge \lambda_d^*(\pi_W(G \cap A)) >0$. Hence $x \in \widehat{B}^{^{\,d}}$.



5. If $x \notin \widehat{A}^{^{\,d}}$ and $x \notin \widehat{B}^{^{\,d}}$, then there exists $G_1,G_2 \in \mathfrak{N}(x)$ such that $\lambda_d^*(\pi_W(G_1 \cap A))=0$ and $\lambda_d^*(\pi_W(G_2 \cap B))=0$ for each $d$-dimensional standard subspace $W$. Choose $G = G_1 \cap G_2 \in \mathfrak{N}(x)$. Thus \begin{center}$\lambda_d^*(\pi_W(G \cap (A \cup B))) \le \lambda_d^*(\pi_W(G_1 \cap A)) + \lambda_d^*(\pi_W(G_2 \cap B))=0$.\end{center} So $x \notin \widehat{A \cup B}^{d}.$ On the other hand, since $A,B \subseteq A \cup B$, we have $\widehat{A}^{^{\,d}} \cup \widehat{B}^{^{\,d}} \subseteq \widehat{A \cup B}^{d}$.

6. For any $x \in \mathbb{R}^n$, $G\in \mathfrak{N}(x)$ and $d$-dimensional standard subspace $W$, $\lambda_d^*(\pi_W( G \cap A)) \le \lambda_d^*(\pi_W(A)) =0$. So $\widehat{A}^{^{\,d}} = \emptyset$. On the other hand, if $\widehat{A}^{^{\,d}} =\emptyset$, then, for each $x \in \mathbb{R}^n$, $x \notin \widehat{A}^{^{\,d}}$. Hence, for each $x \in \mathbb{R}^n$, there exists $G_x \in \mathfrak{N}(x)$ such that $\lambda_d^*(\pi_W(G_x \cap A)) =0$ for each $d$-dimensional standard subspace $W$. But $\{G_x\}_{x \in \mathbb{R}^n}$ covers $\mathbb{R}^n$. Thus there exists a countable subcover: $\{G_1,G_2,\dots\}$. Therefore, \begin{center} $\lambda_d^*(\pi_W(A))\le \displaystyle\sum_{i=1}^{\infty} \lambda_d^*(\pi_W(G_i \cap A))=0$ \end{center} for each $d$-dimensional standard subspace $W$.

7. Observe that, for each $x \in A-\widehat{A}^{^{\,d}}$, $x \notin \widehat{A-\widehat{ A}^{^{\,d}}}^d$ otherwise $x \in \widehat{A}^{^{\,d}}$, a contradiction. Then there exist open sets $\{G_x\}_{x \in A-\widehat{A}^{^{\,d}}}$ such that $\lambda_d^*(\pi_W(G_x \cap (A-\widehat{A}^{^{\,d}}))) =0$ for each $d$-dimensional standard subspace $W$. Notice that $\{G_x\}_{x \in A-\widehat{A}^{^{\,d}}}$ is an open cover of $A-\widehat{A}^{^{\,d}}$. Hence there exists a countable subcover: $\{G_1,G_2,\dots\}$. Thus 
\begin{center}$\lambda_d^*(\pi_W(A-\widehat{A}^{^{\,d}}))\le \displaystyle\sum_{i=1}^{\infty}\lambda_d^*(\pi_W( G_i \cap (A-\widehat{A}^{^{\,d}})))=0$\end{center}
for each $d$-dimensional standard subspace $W$. Therefore, $\widehat{A-\widehat{A}^{^{\,d}}}^{d} = \emptyset$. Observe that \begin{center}$\widehat{A}^{^{\,d}} = \widehat{A \cap \widehat{A}^{^{\,d}}}^{d} \cup \widehat{A-\widehat{A}^{^{\,d}}}^{d} \subseteq \widehat{\widehat{A}^{^{\,d}}}^{d}$.\end{center} On the other hand, we have $\widehat{\widehat{A}^{^{\,d}}}^{d} \subseteq \overline{\widehat{A}^{^{\,d}}}=\widehat{A}^{^{\,d}}.$ Hence $\widehat{A}^{^{\,d}}$ is a $d$-essentially closed set.
\end{proof}

\end{document}